\documentclass[10 pt, a4 paper, twoside]{article}
\usepackage{amsmath,amssymb}
\usepackage{longtable,mathrsfs}
\usepackage{hyperref}
\usepackage{amsmath, amsthm}
\usepackage[mathscr]{euscript}
\newtheorem{thm}{Theorem}

\newtheorem{lem}{Lemma}

\newtheorem{re}{Remark}
\newtheorem{defn}{Definition}

\newcommand{\R}{\mathbb{R}}

{ \setlength{\oddsidemargin}{-2mm}
  \setlength{\evensidemargin}{6mm} }
  \newlength{\titleright}
\allowdisplaybreaks
\begin{document}
\title{\vspace*{-6pc}{\bf Local existence and uniqueness of solutions for fractional differential problem with Hilfer-Hadamard fractional derivative}}

\vspace{1cm}\author{ {\small D. B. Dhaigude$^{1}$
      \footnote{Corresponding author. Email address: sandeeppb7@gmail@.com,  Tel:+91--9421475347}, Sandeep P. Bhairat$^{2}$}\ \\
{\footnotesize \it $^{1,2}$Department of mathematics, Dr. Babasaheb Ambedkar Marathwada University, Aurangabad--431 004 (M.S.) India.}\\}

\date{}

\maketitle

{\vspace*{0.5pc} \hrule\hrule\vspace*{2pc}}

\hspace{-0.8cm}{\bf Abstract}\\
This paper deals with the local existence and uniqueness results for the solution of fractional differential equations with Hilfer-Hadamrd fractional derivative. Using Picard's approximations and generalizing the restrictive conditions imposed on nonlinear function, the iterative scheme for uniformly approximating the solution is constructed.\\

\hspace{-0.8cm}{\it \footnotesize {\bf Keywords:}} {\small Picard iterative technique; singular fractional differential equation, convergence.}
\thispagestyle{empty}
\section{Introduction}
In this work, we are concerned with the Hilfer-Hadamard fractional derivative defined by \cite{kft}
\begin{equation}\label{hh}
(_{H}{\mathscr{D}}_{a^+}^{\alpha,\beta}f)(x)=(_{H}{\mathscr{I}}_{a^+}^{\beta(1-\alpha)}{_{H}{\mathscr{D}}_{a^+}^{\alpha+\beta(1-\alpha)}}f)(x),\qquad 0<\alpha<1,0\leq\beta\leq1,
\end{equation}
where $_{H}{\mathscr{I}}_{a^+}^{\beta(1-\alpha)}$ and ${_{H}{\mathscr{D}}_{a^+}^{\alpha+\beta(1-\alpha)}}$ are the Hadamard fractional integral of order $\beta(1-\alpha)$ and Hadamard fractional derivative of order $\alpha+\beta-\alpha\beta,$ respectively.

Analogous to the Hilfer derivative defined in \cite{hr}, Kassim M D, and N. E. Tatar introduced the Hilfer-Hadamard fractional derivative which interpolates between Hadamard fractional derivative (for $\beta=0$) and Caputo-Hadamard fractional derivative (for $\beta=1$), see \cite{kft}. In \cite{kt}, they established the well-posedness of Cauchy-type problem
\begin{equation}\label{h1}\begin{cases}
&{_{H}{\mathscr{D}}_{a^+}^{\alpha,\beta}x(t)}=f(t,x),\quad t>a>0,\\
&{_{H}{\mathscr{I}}_{a^+}^{1-\gamma}x(a)}=c,\qquad \gamma=\alpha+\beta(1-\alpha),
\end{cases}\end{equation}
where $c\in\R$ and ${_{H}{\mathscr{D}}_{a^+}^{\alpha,\beta}}$ is the HIlfer-Hadamard fractional derivative of order $\alpha (0<\alpha<1)$ and type $\beta (0\leq\beta\leq1),$ in the weighted space of continuous functions $C_{1-\gamma}^{\alpha,\beta}[a,b]$ defined by
\begin{equation}\label{a7}
C_{1-\gamma,\mu}^{\alpha,\beta}[a,b]=\big\{x\in C_{1-\gamma,\log}[a,b]|{_{H}{\mathscr{D}}_{a^+}^{\alpha,\beta}}x\in C_{\mu,\log}[a,b]\big\},\quad 0\leq\mu<1,\gamma=\alpha+\beta(1-\alpha),
\end{equation}
where
\begin{equation}\label{w1}
C_{\gamma,log}[a,b]=\bigg\{g:(a,b]\to\R|\big(\log{\frac{t}{a}}\big)^{\gamma}g(t)\in C[a,b]\bigg\}\quad 0\leq\gamma<1.
\end{equation}
They established the equivalence of initial value problem (IVP) \eqref{h1} with following Volterra integral equation of second kind:
\begin{equation}\label{a9}
x(t)=\frac{c}{\Gamma(\gamma)}\big(\log{\frac{t}{a}}\big)^{\gamma-1}+\frac{1}{\Gamma(\alpha)}\int_{a}^{t}\big(\log{\frac{t}{s}}\big)^{\alpha-1}f(s,x(s))\frac{ds}{s}, \quad t>a, \, c\in\R,
\end{equation}
and using the Banach fixed point theorem, following existence result for IVP \eqref{h1} is proved in \cite{kt}.
\begin{thm}\cite{kt}
Let $\gamma=\alpha+\beta-\alpha\beta$ where $(0<\alpha<1,0\leq\beta\leq1).$ Assume that $f:(a,b]\times\R\to\R,(a>0),$ is a function such that $f[\cdot,x(\cdot)]\in{C_{\mu,\log}[a,b]}$ for any $x\in{C_{\mu,\log}[a,b]}$ with $1-\gamma\leq\mu<1-\beta(1-\alpha)$ and is Lipschitz continuous with respect to its second variable. Then, there exists a unique solution $x$ for the Cauchy-type problem \eqref{h1} in the space $C_{1-\gamma,\mu}^{\alpha,\beta}[a,b].$
\end{thm}
We also point out that, when $f(t,x(t))\geq(\log{\frac{t}{a}})^{\mu}|x(t)|^{m}$ for some $m>1$ and $\mu\in\R,$ a nonexistence for global solutions of problem
\begin{equation}\label{h2}\begin{cases}
&{_{H}{\mathscr{D}}_{a^+}^{\alpha,\beta}x(t)}\geq(\log{\frac{t}{a}})^{\mu}|x(t)|^{m},\quad t>a>0,m>1,\mu\in\R,\\
&{_{H}{\mathscr{I}}_{a^+}^{1-\gamma}x(a)}=c,\qquad \gamma=\alpha+\beta(1-\alpha),
\end{cases}
\end{equation}
is proved in the following theorem.
\begin{thm}\cite{kft}
Assume that $\mu\in\R$ and $m<(1+\mu)/(1-\gamma).$ Then, the problem \eqref{h2} does not admit global nontrivial solutions in $C_{1-\gamma,\log}^{\gamma}[a,b],$ where $C_{1-\gamma,\log}^{\gamma}[a,b]=\{y\in C_{1-\gamma,\log}[a,b]:{_{H}{\mathscr{D}}_{a^+}^{\gamma}}C_{1-\gamma,\log}[a,b]\}$ and $c\geq0.$
\end{thm}
Recently, in a survey paper \cite{abl}, Said Abbas, et.al. obtained the results concerning the existence and uniqueness of weak solutions for some classes of Hadamard and Hilfer fractional differential equations. Further, some attractivity and Ulam stability results are obtained by applying the fixed point theory. Authors in \cite{db1}-\cite{db3} obtained the existence, uniqueness and continuations results by using both successive approximations and fixed point techniques for the solution of fractional IVP involving Hilfer fractional derivative defined in \cite{hr}.

We find that the existence and uniqueness results were proved, but the iterative scheme for uniformly approximating the solution of IVP \eqref{h1} was not given in (Theorem 21 \cite{kt}). Generally, finding the solution of nonlinear fractional differential equation is not an easy task. So the numerical treatment for such problems practically more sounds.

Motivated by this work, to avoid the ambiguity of fixed point theory, we adopted the method of successive approximations.
In this paper, we will study the IVP for fractional differential equation
\begin{equation}\label{s1}\begin{cases}
&{_{H}{\mathscr{D}}_{1}^{\alpha,\beta}}x(t)=f(t,x),\qquad \,\, 0<\alpha<1,0\leq\beta\leq1,\\
&\lim_{t\to{1}}\big(\log{t}\big)^{1-\gamma}x(t)=x_0,\quad \gamma=\alpha+\beta(1-\alpha).
\end{cases}
\end{equation}
Using some well-known convergence criterion and Picard's sequence functions \cite{km},\cite{yy}, we establish the existence and uniqueness results of IVP \eqref{s1}. The computable iterative scheme as well as the uniform convergence criterion for the solution are also developed. Note that the initial value considered in IVP \eqref{s1} is more suitable than that of considered in IVP \eqref{h1} and nonlinear function $f$ may be singular at $t=1.$

The rest of the paper is organised as follows: the next section covers the useful prerequisites. The main results are proved in section 3. Conclusion is given in the last section.

\section{Preliminaries}
We need the following basic definitions and properties from fractional calculus in the sequel, see \cite{kst}.

Let the Euler's gamma function $\Gamma(\cdot)$ defined by \cite{lv}
\begin{equation*}
\Gamma(x)=\int_{0}^{+\infty}s^{x-1}e^{-s}ds,\quad x>0.
\end{equation*}
\begin{defn}\cite{kst}
Let $(1,b),1<b\leq\infty,$ be a finite or infinite interval of the half-axis ${\R}^+$ and let $\alpha>0.$ The left-sided Hadamard fractional integral ${_{H}{\mathscr{I}}_{1}^{\alpha}f}$ of order $\alpha>0$ is defined by
\begin{equation}\label{hi}
(_{H}{\mathscr{I}}_{1}^{\alpha}f)(t)=\frac{1}{\Gamma(\alpha)}\int_{1}^{t}(\log{t})^{\alpha-1}\frac{f(s)ds}{s},\quad 1<t<b,
\end{equation}
provided that the integral exists. When $\alpha=0,$ we set ${_{H}{\mathscr{I}}_{1}^{0}f=f.}$
\end{defn}
\begin{defn}\cite{kst}
The left-sided Hadamard fractional derivative of order $\alpha(0\leq\alpha<1)$ on $(1,b)$ is defined by
\begin{equation}\label{hd}
(_{H}{\mathscr{D}}_{1}^{\alpha}f)(t)=\delta(_{H}{\mathscr{I}}_{1}^{1-\alpha}f)(t),\qquad 1<t<b,
\end{equation}
where $\delta=t(d/dt).$ In particular, when $\alpha=0$ we have $_{H}{\mathscr{D}}_{1}^{0}f=f.$
\end{defn}
\begin{defn}\cite{kst}
Let $(1,b)$ be a finite interval of the half-axis ${\R}^{+}.$ The fractional derivative $_{H}^{c}{\mathscr{D}}_{1}^{\alpha}f$ of order $\alpha(0<\alpha<1)$ on $(1,b)$ defined by
\begin{equation}\label{chd}
_{H}^{c}{\mathscr{D}}_{1}^{\alpha}f={_{H}{\mathscr{I}}_{1}^{1-\alpha}\delta f},
\end{equation}
is called the left-sided Hadamard-Caputo fractional derivative of order $\alpha$ of a function $f.$
\end{defn}
\begin{defn}\cite{kt}
The left-sided Hilfer-Hadamard fractional derivative of order $\alpha(0<\alpha<1)$ and type $\beta(0\leq\beta\leq1)$ with respect to $t$ is defined by
\begin{equation}\label{hh}
(_{H}{\mathscr{D}}_{1}^{\alpha,\beta}f)(t)=(_{H}{\mathscr{I}}_{1}^{\beta(1-\alpha)}{_{H}{\mathscr{D}}_{1}^{\alpha+\beta(1-\alpha)}}f)(t)
\end{equation}
of functions $f$ for which the expression on the right hand side exists, where ${_{H}{\mathscr{D}}_{1}^{\alpha+\beta(1-\alpha)}}$ is the Hadamard fractional derivative.
\end{defn}
\begin{lem}\cite{kst}
If $\alpha>0,\beta>0$ and $1<b<\infty,$ then
\begin{align}\label{pr}
\big({_{H}{\mathscr{I}}_{1}^{\alpha}}\big(\log{s}\big)^{\beta-1}\big)(t)&=\frac{\Gamma(\beta)}{\Gamma(\alpha+\beta)}\big(\log{t}\big)^{\beta+\alpha-1},\\
\big({_{H}{\mathscr{D}}_{1}^{\alpha}}\big(\log{s}\big)^{\beta-1}\big)(t)&=\frac{\Gamma(\beta)}{\Gamma(\beta-\alpha)}\big(\log{t}\big)^{\beta-\alpha-1}.
\end{align}
\end{lem}
\noindent Following lemma plays vital role in the proof of the main results, the detailed proof can be found in \cite{pi}.
\begin{lem}\cite{pi} Suppose that $x>0.$ Then $\Gamma(x)=\lim_{m\to+\infty}\frac{m^{x}m!}{x(x+1)(x+2)\cdots(x+m)}.$
\end{lem}
We denote $D=[1,1+h], D_{h}=(1,1+h], I=(1,1+l],J=[1,1+l]$, $E=\{x:|x(\log{t})^{1-\gamma}-x_0|\leq b\}$ for $h>0,b>0$ and $t\in{D_h}.$ A function $x(t)$ is said to be a solution of IVP \eqref{s1} if there exist $l>0$ such that $x\in C^{0}(I)$ satisfies the equation  $_{H}{\mathscr{D}}_{1}^{\alpha,\beta}x(t)=f(t,x)$ almost everywhere on $I$ alongwith the condition $\lim_{t\to{1}}{(\log{t})}^{1-\gamma}x(t)=x_0.$ To construct the main results, the following hypotheses are considered:
\begin{description}
\item[(H1)]  $(t,x)\to f(t,(\log{t})^{\gamma-1}x(t))$ is defined on ${D}_{h}\times E$ satisfies:
\begin{itemize}
\item[(i)] $x\to f(t,(\log{t})^{\gamma-1}x(t))$ is continuous on $E$ for all $t\in{D_{h}}$,\\
 $t\to f(t,(\log{t})^{\gamma-1}x(t))$ is measurable on $D_{h}$ for all $x\in E;$\
\item[(ii)] there exist $k>(\beta(1-\alpha)-1)$ and $M\geq0$ such that the relation $|f(t,(\log{t})^{\gamma-1}x(t))|\leq M(\log{t})^{k}$ holds for all $t\in D_{h}$ and $x\in E,$
\end{itemize}
\item[(H2)] there exists $A>0$ such that $|f(t,(\log{t})^{\gamma-1}x_1(t))-f(t,(\log{t})^{\gamma-1}x_2(t))|$ $\leq A(\log{t})^{k}|x_1-x_2|,$ for all $t\in I$ and $x_1,x_2\in E.$
\end{description}
\begin{re}
In hypothesis \textbf{(H1)}, if $(\log{t})^{-k}f(t,(\log{t})^{\gamma-1}x(t))$ is continuous on $D\times E,$ one may choose $M=\max_{t\in J}(\log{t})^{-k}f(t,(\log{t})^{\gamma-1}x(t))$ continuous on ${D_h\times E}$ for all $x\in E.$
\end{re}

\section{Main results}
In this section, we state and prove the existence and uniqueness results for IVP \eqref{s1} concerned with above defined hypotheses. We present the iterative scheme for approximating such a unique solution.

For brevity let us choose $l=\min\bigg{\{h,{\big(\frac{b}{M}\frac{\Gamma(\alpha+k+1)}{\Gamma(k+1)}\big)}^{\frac{1}{\mu+k}}\bigg\}},\, \mu=1-\beta(1-\alpha).$
\begin{lem}
Suppose that \textbf{(H1)} holds. Then $x:J\to\R$ is a solution of IVP \eqref{s1} if and only if $x:I\to\R$ is a solution of the Volterra integral equation of second kind:
\begin{equation}\label{s2}
x(t)=x_0{\big(\log{t}\big)}^{\gamma-1}+\frac{1}{\Gamma(\alpha)}\int_{1}^{t}{\big(\log{\frac{t}{s}}\big)}^{\alpha-1}f(s,x(s))\frac{ds}{s},\quad t>1.
\end{equation}
\end{lem}
\begin{proof} First we suppose that $x:I\to\R$ is a solution of IVP \eqref{s1}. Then $|{\big(\log{t}\big)}^{1-\gamma}x(t)-x_0|\leq b$ for all $t\in I.$ From \textbf{(H1)}, there exists a $k>(\beta(1-\alpha)-1)$ and $M\geq0$ such that
\begin{equation*}
  |f(t,x(t))|=|f(t,{(\log{t})}^{\gamma-1}{(\log{t})}^{1-\gamma}x(t))|\leq M{(\log{t})}^{k},\qquad \text{for all}\quad t\in I.
\end{equation*}
We have
\begin{align*}
\bigg{|}\frac{1}{{\Gamma(\alpha)}}\int_{1}^{t}{{\big(\log{\frac{t}{s}}\big)}^{\alpha-1}}f(s,x(s))\frac{ds}{s}\bigg{|}&\leq \frac{1}{{\Gamma(\alpha)}}\int_{1}^{t}{{\big(\log{\frac{t}{s}}\big)}^{\alpha-1}}M{{(\log{s})}^{k}}\frac{ds}{s}\\
&=M{(\log{t})}^{\alpha+k}\frac{\Gamma(k+1)}{\Gamma(\alpha+k+1)}.
\end{align*}
Clearly,
\begin{equation*}
\lim_{t\to1}{\big(\log{t}\big)}^{1-\gamma}\frac{1}{{\Gamma(\alpha)}}\int_{1}^{t}{{\big(\log{\frac{t}{s}}\big)}^{\alpha-1}}f(s,x(s))\frac{ds}{s}=0.
\end{equation*}
It follows that
\begin{equation*}
x(t)=x_0{\big(\log{t}\big)}^{\gamma-1}+\frac{1}{{\Gamma(\alpha)}}\int_{1}^{t}{{\big(\log{\frac{t}{s}}\big)}^{\alpha-1}}f(s,x(s))\frac{ds}{s},\quad t\in I.
\end{equation*}
Since $k>(\beta(1-\alpha)-1),$ then $x\in{C^{0}(I)}$ is a solution of integral equation \eqref{s2}.

Conversely, it is easy to see that $x:I\to\R$ is a solution of integral equation \eqref{s2} implies that $x$ is a solution of IVP \eqref{s1} defined on $J.$ This completes the proof.
\end{proof}

To prove further main results, we choose a Picard function sequence as follows:
\begin{equation}\label{pfc}\begin{split}
  \phi_0(t)&=x_0{(\log{t})}^{\gamma-1},\qquad t\in I, \\
  \phi_n(t)=\phi_0(t)+&\frac{1}{\Gamma(\alpha)}\int_{1}^{t}{\big(\log{\frac{t}{s}}\big)}^{\alpha-1}f(s,\phi_{n-1}(s))\frac{ds}{s},\quad t\in I,\quad n=1,2,\cdots.
\end{split}\end{equation}
\begin{lem}
Suppose that \textbf{(H1)} holds. Then $\phi_n$ is continuous on $I$ and satisfies $|{(\log{t})}^{1-\gamma}\phi_n(t)-x_0|\leq b.$
\end{lem}
\begin{proof} From \textbf{(H1)}, clearly $|f(t,{(\log{t})}^{\gamma-1}x)|\leq M{(\log{t})}^{k}$ for all $t\in{D_h}$ and $|x{(\log{t})}^{1-\gamma}-x_0|\leq b.$ For $n=1,$ we have
\begin{equation}\label{l1}
\phi_1(t)=x_0{(\log{t})}^{\gamma-1}+\frac{1}{\Gamma(\alpha)}\int_{1}^{t}{{\big(\log{\frac{t}{s}}\big)}^{\alpha-1}}f(s,\phi_{0}(s))\frac{ds}{s}.
\end{equation}
Then
\begin{equation*}
\bigg{|}\frac{1}{\Gamma(\alpha)}\int_{1}^{t}{{\big(\log{\frac{t}{s}}\big)}^{\alpha-1}}f(s,\phi_0(s))\frac{ds}{s}\bigg{|}\leq \frac{1}{\Gamma(\alpha)}\int_{1}^{t}{{\big(\log{\frac{t}{s}}\big)}^{\alpha-1}}M{(\log{s}\big)}^{k}\frac{ds}{s}=M{(\log{t}\big)}^{\alpha+k}\frac{\Gamma(k+1)}{\Gamma(\alpha+k+1)}.
\end{equation*}
This implies $\phi_1\in{C^{0}(I)}$ and from equation \eqref{l1}, we get
\begin{equation}\label{l2}
  |{(\log{t})}^{1-\gamma}\phi_1(t)-x_0|\leq{(\log{t})}^{1-\gamma}M{(\log{t})}^{\alpha+k}\frac{\Gamma(k+1)}{\Gamma(\alpha+k+1)}\leq Ml^{\alpha+k+1-\gamma}\frac{\Gamma(k+1)}{\Gamma(\alpha+k+1)}.
\end{equation}
Now by induction hypothesis, suppose that $\phi_n\in{C^{0}(J)}$ and $|{(\log{t})}^{1-\gamma}\phi_n(t)-x_0|\leq b$ for all $t\in J.$ We have
\begin{equation}\label{l3}
\phi_{n+1}(t)=x_0{(\log{t})}^{\gamma-1}+\frac{1}{\Gamma(\alpha)}\int_{1}^{t}{{\big(\log{\frac{t}{s}}\big)}^{\alpha-1}}f(s,\phi_{n}(s))\frac{ds}{s}.
\end{equation}
From above discussion, we obtain $\phi_{n+1}(t)\in {C^{0}(I)}$ and from equation \eqref{l3}, we have
\begin{align*}
   |{(\log{t}\big)}^{1-\gamma}\phi_{n+1}(t)-x_0|&\leq {(\log{t})}^{1-\gamma}\frac{1}{\Gamma(\alpha)}\int_{1}^{t}{{\big(\log{\frac{t}{s}}\big)}^{\alpha-1}}M{(\log{s}\big)}^{k}\frac{ds}{s}\\
  &=M{(\log{t})}^{\alpha+k+1-\gamma}\frac{\Gamma(k+1)}{\Gamma(\alpha+k+1)} \\
  &\leq Ml^{\alpha+k+1-\gamma}\frac{\Gamma(k+1)}{\Gamma(\alpha+k+1)}\leq b.
\end{align*}
Thus, the result is true for $n+1$ and holds the induction hypotheses. Therefore, by the mathematical induction principle, the result is true for all $n.$ The proof is complete.
\end{proof}

\begin{thm}
Suppose \textbf{(H1)}-\textbf{(H2)} holds. Then $\{{(\log{t})}^{1-\gamma}\phi_n(t)\}$ is uniformly convergent sequence on $J.$
\end{thm}
\begin{proof} Consider the series
\begin{equation*}
{{(\log{t})}^{1-\gamma}\phi_0(t)}+{{(\log{t})}^{1-\gamma}[\phi_1(t)-\phi_0(t)]}+\cdots+{{(\log{t})}^{1-\gamma}[\phi_n(t)-\phi_{n-1}(t)]}+\cdots,\quad t\in J.
\end{equation*}
By relation \eqref{l2} driven in the proof of Lemma 4 above,
\begin{equation*}
{(\log{t})}^{1-\gamma}|\phi_1(t)-\phi_0(t)|\leq M{(\log{t})}^{\alpha+k+1-\gamma}\frac{\Gamma(k+1)}{\Gamma(\alpha+k+1)}, \qquad t\in J.
\end{equation*}
From Lemma 4,
\begin{align*}
{(\log{t})}^{1-\gamma}|\phi_2(t)&-\phi_1(t)|\leq{(\log{t})}^{1-\gamma}\frac{1}{\Gamma(\alpha)}\int_{1}^{t}{{\big(\log{\frac{t}{s}}\big)}^{\alpha-1}}|f(s,\phi_1(s))-f(s,\phi_0(s))|\frac{ds}{s}\\
=&{(\log{t}\big)}^{1-\gamma}\frac{1}{\Gamma(\alpha)}\int_{1}^{t}{{\big(\log{\frac{t}{s}}\big)}^{\alpha-1}}\big|f\big(s,{(\log{s})}^{\gamma-1}{(\log{s})}^{1-\gamma}\phi_1(s)\big)\\
&\hspace{3cm}-f\big(s,{(\log{s})}^{\gamma-1}{(\log{s})}^{1-\gamma}\phi_0(s)\big)\big|\frac{ds}{s}\\
\leq&{(\log{t})}^{1-\gamma}\frac{1}{\Gamma(\alpha)}\int_{1}^{t}{{\big(\log{\frac{t}{s}}\big)}^{\alpha-1}}A{(\log{s})}^{k}\big|{(\log{s})}^{1-\gamma}\phi_1(s)-{(\log{s})}^{1-\gamma}\phi_0(s)\big|\frac{ds}{s}\\
\leq&{(\log{t})}^{1-\gamma}\frac{1}{\Gamma(\alpha)}\int_{1}^{t}{{\big(\log{\frac{t}{s}}\big)}^{\alpha-1}}A{(\log{s})}^{k}\big[{(\log{s})}^{1-\gamma}|\phi_1(s)-\phi_0(s)|\big]\frac{ds}{s}\\
\leq&{(\log{t})}^{1-\gamma}\frac{1}{\Gamma(\alpha)}\int_{1}^{t}{{\big(\log{\frac{t}{s}}\big)}^{\alpha-1}}A{(\log{s})}^{k}\big[M{(\log{s})}^{\alpha+k+1-\gamma}\frac{\Gamma(k+1)}{\Gamma(\alpha+k+1)}\big]\frac{ds}{s}\\
=&AM\frac{\Gamma(k+1)}{\Gamma(\alpha+k+1)}\frac{\Gamma(\alpha+2k+2-\gamma)}{\Gamma(2\alpha+2k+2-\gamma)}{(\log{t})}^{2(\alpha+k+1-\gamma)}.
\end{align*}
Now suppose that
\begin{equation*}
 {(\log{t})}^{1-\gamma}|\phi_{n+1}(t)-\phi_n(t)|\leq A^{n}M{(\log{t})}^{(n+1)(\alpha+k+1-\gamma)}\prod_{i=0}^{n}\frac{\Gamma((i+1)k+i(\alpha+1-\gamma)+1)}{\Gamma((i+1)(\alpha+k)+i(1-\gamma)+1)}.
\end{equation*}
We have
\begin{align*}
{(\log{t})}^{1-\gamma}|\phi_{n+2}(t)-\phi_{n+1}(t)|&\leq{(\log{t})}^{1-\gamma}\frac{1}{\Gamma(\alpha)}\int_{1}^{t}{{\big(\log{\frac{t}{s}}\big)}^{\alpha-1}}|f(s,\phi_{n+1}(s))-f(s,\phi_n(s))|\frac{ds}{s}\\
&={(\log{t})}^{1-\gamma}\frac{1}{\Gamma(\alpha)}\int_{1}^{t}{{\big(\log{\frac{t}{s}}\big)}^{\alpha-1}}\big|f\big(s,{(\log{s})}^{\gamma-1}{(\log{s})}^{1-\gamma}\phi_{n+1}(s)\big)\\
&\hspace{3cm}-f\big(s,{(\log{s})}^{\gamma-1}{(\log{s})}^{1-\gamma}\phi_n(s)\big)\big|\frac{ds}{s}\\
{(\log{t})}^{1-\gamma}|\phi_{n+2}(t)-\phi_{n+1}(t)|&\leq{(\log{t})}^{1-\gamma}\frac{1}{(\Gamma(\alpha))}\int_{1}^{t}{{\big(\log{\frac{t}{s}}\big)}^{\alpha-1}}A{(\log{s})}^{k}\big[{(\log{s})}^{1-\gamma}|\phi_{n+1}(s)-\phi_n(s)|\big]\frac{ds}{s}\\
&\leq{(\log{t})}^{1-\gamma}\frac{1}{\Gamma(\alpha)}\int_{1}^{t}{{\big(\log{\frac{t}{s}}\big)}^{\alpha-1}}A{(\log{s})}^{k}\big[A^{n}M{(\log{s})}^{(n+1)(\alpha+k+1-\gamma)}\\
&\hspace{2.5cm}\times\prod_{i=0}^{n}\frac{\Gamma((i+1)k+i(\alpha+1-\gamma)+1)}{\Gamma((i+1)(\alpha+k)+i(1-\gamma)-1)}\big]\frac{ds}{s}\\
&={A^{n+1}M{(\log{t})}^{(n+2)(\alpha+k+1-\gamma)}}\prod_{i=0}^{n+1}\frac{\Gamma((i+1)k+i(\alpha+1-\gamma)+1)}{\Gamma((i+1)(\alpha+k)+i(1-\gamma)+1)}
\end{align*}
which follows that the result is true for $n+1.$ Using principle of mathematical induction, we get
\begin{equation}\label{l4}
{(\log{t})}^{1-\gamma}|\phi_{n+2}(t)-\phi_{n+1}(t)|\leq A^{n+1}Ml^{(n+2)(\alpha+k+1-\gamma)}\prod_{i=0}^{n+1}\frac{\Gamma((i+1)k+i(\alpha+1-\gamma)+1)}{\Gamma((i+1)(\alpha+k)+i(1-\gamma)+1)}.
\end{equation}
Consider
\begin{equation*}
\sum_{n=1}^{\infty}u_n=\sum_{n=1}^{\infty}MA^{n+1}l^{(n+2)(\alpha+k+1-\gamma)}\prod_{i=0}^{n+1}\frac{\Gamma((i+1)k+i(\alpha+1-\gamma)+1)}{\Gamma((i+1)(\alpha+k)+i(1-\gamma)+1)}.
\end{equation*}
We have
\begin{align*}
\frac{u_{n+1}}{u_n}&=\frac{MA^{n+2}l^{(n+3)(\alpha+k+1-\gamma)}\prod_{i=0}^{n+2}\frac{\Gamma((i+1)k+i(\alpha+1-\gamma)+1)}{\Gamma((i+1)(\alpha+k)+i(1-\gamma)+1)}}{MA^{n+1}l^{(n+2)(\alpha+k+1-\gamma)}\prod_{i=0}^{n+1}\frac{\Gamma((i+1)k+i(\alpha+1-\gamma)+1)}{\Gamma((i+1)(\alpha+k)+i(1-\gamma)+1)}}\\
&=Al^{\alpha+k+1-\gamma}\frac{\Gamma((n+3)k+(n+2)(\alpha+1-\gamma)+1)}{\Gamma((n+3)(k+\alpha)+(n+2)(1-\gamma)+1)}.
\end{align*}
Using Lemma 2, we have
\begin{align*}
\frac{u_{n+1}}{u_n}&=Al^{\alpha+k+1-\gamma}\frac{\lim_{m\to\infty}\frac{m^{(n+3)k+(n+2)(\alpha+1-\gamma)+1}m!}{((n+3)k+(n+2)(\alpha+1-\gamma)+1)\cdots((n+3)k+(n+2)(\alpha+1-\gamma)+m+1)}}{\lim_{m\to\infty}\frac{m^{(n+3)(k+\alpha)+(n+2)(1-\gamma)+1}m!}{((n+3)(k+\alpha)+(n+2)(1-\gamma)+1)\cdots((n+3)(k+\alpha)+(n+2)(1-\gamma)+m+1)}}
\end{align*}
$\hspace{2.5cm}=Al^{\alpha+k+1-\gamma}[\lim_{m\to\infty}m^{-\alpha}\frac{((n+3)(k+\alpha)+(n+2)(1-\gamma)+1)\cdots((n+3)(k+\alpha)+(n+2)(1-\gamma)+m+1)}
{((n+3)k+(n+2)(\alpha+1-\gamma)+1)\cdots((n+3)k+(n+2)(\alpha+1-\gamma)+m+1)}].$\\ \\
It is easy to see that $$\frac{((n+3)(k+\alpha)+(n+2)(1-\gamma)+1)\cdots((n+3)(k+\alpha)+(n+2)(1-\gamma)+m+1)}
{((n+3)k+(n+2)(\alpha+1-\gamma)+1)\cdots((n+3)k+(n+2)(\alpha+1-\gamma)+m+1)}$$ is bounded for all $m,n.$ Thus $\lim_{n\to\infty}\frac{u_{n+1}}{u_n}=0.$ This implies $\sum_{n=1}^{\infty}u_n$ is convergent. Hence the series
\begin{equation*}
{{(\log{t}\big)}^{1-\gamma}\phi_0(t)}+{{(\log{t})}^{1-\gamma}[\phi_1(t)-\phi_0(t)]}+\cdots+{{(\log{t})}^{1-\gamma}[\phi_n(t)-\phi_{n-1}(t)]}+\cdots
\end{equation*}
is uniformly convergent for $t\in J.$ Therefore $\{{(\log{t})}^{1-\gamma}\phi_n(t)\}$ is uniformly convergent sequence on $J.$
\end{proof}

\begin{thm}
Suppose that \textbf{(H1)}-\textbf{(H2)} holds. Then $\phi(t)={(\log{t})}^{\gamma-1}\lim_{n\to\infty}{(\log{t})}^{1-\gamma}\phi_n(t)$ is a unique continuous solution of integral equation \eqref{s2} defined on $J.$
\end{thm}
\begin{proof} Since $\phi(t)={(\log{t})}^{\gamma-1}\lim_{n\to\infty}{(\log{t})}^{1-\gamma}\phi_n(t)$ on $J,$ and by Lemma 2, ${(\log{t})}^{1-\gamma}|\phi(t)-x_0|\leq b.$ Then
\begin{align*}
  |f(t,\phi_{n}(t))-f(t,\phi(t))|\leq A{(\log{t})}^{k}&|\phi_{n}(t)-\phi(t)|,\quad t\in I,\\
  {(\log{t})}^{-k}|f(t,\phi_{n}(t))-f(t,\phi(t))|&\leq A|\phi_{n}(t)-\phi(t)|\to0
\end{align*}
uniformly as $n\to\infty$ on $I.$ Therefore
\begin{align*}
{(\log{t})}^{1-\gamma}\phi(t)&=\lim_{n\to\infty}\phi_{n}(t)\\
&=x_0+{(\log{t})}^{1-\gamma}\lim_{n\to\infty}\frac{1}{\Gamma(\alpha)}\int_{1}^{t}{{\big(\log{\frac{t}{s}}\big)}^{\alpha-1}}{(\log{s})}^{k}\big({(\log{s})}^{-k}f(s,\phi_{n-1}(s))\big)\frac{ds}{s}\\
&=x_0+{(\log{t})}^{1-\gamma}\frac{1}{\Gamma(\alpha)}\int_{1}^{t}{{\big(\log{\frac{t}{s}}\big)}^{\alpha-1}}{(\log{s})}^{k}\lim_{n\to\infty}\big({(\log{s})}^{-k}f(s,\phi_{n-1}(s))\big)\frac{ds}{s}\\
&=x_0+{(\log{t})}^{1-\gamma}\frac{1}{\Gamma(\alpha)}\int_{1}^{t}{{\big(\log{\frac{t}{s}}\big)}^{\alpha-1}}f(s,\phi(s))\frac{ds}{s}.
\end{align*}
Then $\phi$ is a continuous solution of integral equation \eqref{s2} defined on $J.$

Now we will prove uniqueness of solution $\phi(t).$ For this, suppose that $\psi(t)$ defined on $I$ is also a solution of integral equation \eqref{s2}. Then ${(\log{t})}^{1-\gamma}|\psi(t)|\leq b$ for all $t\in I$ and
\begin{equation*}
\psi(t)=x_0{(\log{t})}^{\gamma-1}+\frac{1}{\Gamma(\alpha)}\int_{1}^{t}{{\big(\log{\frac{t}{s}}\big)}^{\alpha-1}}f(s,\phi(s))\frac{ds}{s},\quad t\in I.
\end{equation*}
It is sufficient to prove $\phi(t)\equiv\psi(t)$ on $I.$ From \textbf{(H1)}, there exists a $k>(\beta(1-\alpha)-1)$ and $M\geq0$ such that
\begin{equation*}
  |f(t,\psi(t))|=\big|f\big(t,{(\log{t})}^{\gamma-1}{(\log{t})}^{1-\gamma}\psi(t)\big)\big|\leq M{(\log{t})}^{k},
\end{equation*}
for all $t\in I.$ Therefore
\begin{align*}
{(\log{t})}^{1-\gamma}|\phi_{0}(t)-\psi(t)|=&{(\log{t})}^{1-\gamma}\bigg|\frac{1}{\Gamma(\alpha)}\int_{1}^{t}{{\big(\log{\frac{t}{s}}\big)}^{\alpha-1}}f(s,\psi(s))\frac{ds}{s}\bigg|\\
&\leq{(\log{t})}^{1-\gamma}\frac{1}{\Gamma(\alpha)}\int_{1}^{t}{{\big(\log{\frac{t}{s}}\big)}^{\alpha-1}}M{(\log{s})}^{k}\frac{ds}{s}\\
&=M{(\log{t})}^{\alpha+k+1-\gamma}\frac{\Gamma(k+1)}{\Gamma(\alpha+k+1)}.
\end{align*}
Furthermore
\begin{align*}
{(\log{t})}^{1-\gamma}|\phi_{1}(t)-\psi(t)|=&{(\log{t})}^{1-\gamma}\bigg|\frac{1}{\Gamma(\alpha)}\int_{1}^{t}{{\big(\log{\frac{t}{s}}\big)}^{\alpha-1}}[f(s,\phi_0(s))-f(s,\psi(s))]\frac{ds}{s}\bigg|\\
&\leq AM \frac{\Gamma(k+1)}{\Gamma(\alpha+k+1)}\frac{\Gamma(\alpha+2k+2-\gamma)}{\Gamma(2\alpha+2k+2-\gamma)}{(\log{t})}^{2(\alpha+k+1-\gamma)}.
\end{align*}
By the induction hypothesis, we suppose that
\begin{equation*}
{(\log{t})}^{1-\gamma}|\phi_{n}(t)-\psi(t)|\leq A^{n}M{(\log{t})}^{(n+1)(\alpha+k+1-\gamma)}\prod_{i=0}^{n}\frac{\Gamma((i+1)k+i(\alpha+1-\gamma)+1)}{\Gamma((i+1)(\alpha+k)+i(1-\gamma)+1)}.
\end{equation*}
Then
\begin{align*}
{(\log{t})}^{1-\gamma}|\phi_{n+1}(t)-\psi(t)|\leq&{(\log{t})}^{1-\gamma}\bigg|\frac{1}{\Gamma(\alpha)}\int_{1}^{t}{{\big(\log{\frac{t}{s}}\big)}^{\alpha-1}}[f(s,\phi_n(s))-f(s,\psi(s))]\frac{ds}{s}\bigg|\\
\leq& A^{n+1}M{(\log{t})}^{(n+2)(\alpha+k+1-\gamma)}\prod_{i=0}^{n+1}\frac{\Gamma((i+1)k+i(\alpha+1-\gamma)+1)}{\Gamma((i+1)(\alpha+k)+i(1-\gamma)+1)}\\
\leq& A^{n+1}Ml^{(n+2)(\alpha+k+1-\gamma)}\prod_{i=0}^{n+1}\frac{\Gamma((i+1)k+i(\alpha+1-\gamma)+1)}{\Gamma((i+1)(\alpha+k)+i(1-\gamma)+1)}.
\end{align*}
Using the same arguments used in the proof of Theorem 3, we obtain the series
\begin{equation*}
  \sum_{n=1}^{\infty}A^{n+1}Ml^{(n+2)(\alpha+k+1-\gamma)}\prod_{i=0}^{n+1}\frac{\Gamma((i+1)k+i(\alpha+1-\gamma)+1)}{\Gamma((i+1)(\alpha+k)+i(1-\gamma)+1)}
\end{equation*}
is convergent. Thus $A^{n+1}Ml^{(n+2)(\alpha+k+1-\gamma)}\prod_{i=0}^{n+1}\frac{\Gamma((i+1)k+i(\alpha+1-\gamma)+1)}{\Gamma((i+1)(\alpha+k)+i(1-\gamma)+1)}\to0$ as $n\to\infty.$ Also we observe that $\lim_{n\to\infty}{(\log{t})}^{1-\gamma}\phi_n(t)={(\log{t})}^{1-\gamma}\psi(t)$ uniformly on $J.$ Thus $\phi(t)\equiv\psi(t)$ on $I.$
\end{proof}

\begin{thm}
Suppose that \textbf{(H1)}-\textbf{(H2)} holds. Then the IVP \eqref{s1} has a unique continuous solution $\phi$ defined on $I$ and $\phi(t)={(\log{t})}^{\gamma-1}\lim_{n\to\infty}{(\log{t})}^{1-\gamma}\phi_{n}(t)$ with $\phi_0(t)$ and $\phi_n(t)$ defined by \eqref{pfc}.
\end{thm}
\begin{proof} From Lemma 3 and Theorem 3, we can easily obtain that $\phi(t)={(\log{t})}^{\gamma-1}\lim_{n\to\infty}{(\log{t})}^{1-\gamma}\phi_n(t)$ is a unique continuous solution of IVP \eqref{s1} defined on $I$. Thus the proof is ended here.
\end{proof}
\section{Concluding remarks} We considered a new class of IVP for fractional differential problems with Hilfer-Hadamard fractional derivative. A new criteria for the local existence and uniqueness of solution is discussed. Then uniform convergence of such a local solution is obtained with Picard iterative method and a computable sequences are given for approximating the solutions.

\end{document}